\documentclass[11pt]{amsart}
\usepackage{graphicx}
\usepackage[active]{srcltx}
 \makeatletter
\renewcommand*\subjclass[2][2000]{%
  \def\@subjclass{#2}%
  \@ifundefined{subjclassname@#1}{%
    \ClassWarning{\@classname}{Unknown edition (#1) of Mathematics
      Subject Classification; using '1991'.}%
  }{%
    \@xp\let\@xp\subjclassname\csname subjclassname@#1\endcsname
  }%
}
 \makeatother
\usepackage{enumerate,url,amssymb,  mathrsfs}%, pdfsync}% ,showkeys, pdfsync}

\newtheorem{theorem}{Theorem}[section]

\newtheorem*{lemma*}{Lemma}

\newtheorem{corollary}[theorem]{Corollary}

\theoremstyle{definition}

\theoremstyle{remark}
\newtheorem{remark}[theorem]{Remark}

\numberwithin{equation}{section}

\def\XXint#1#2#3{{\setbox0=\hbox{$#1{#2#3}{\int}$}
\vcenter{\hbox{$#2#3$}}\kern-.5\wd0}}

\def\le{\leqslant}
\def\ge{\geqslant}
\begin{document}

\title{Schwarz lemma for harmonic mappings in the unit ball}

\author[Kalaj]{David Kalaj}
\address{University of Montenegro, Faculty of Natural Sciences and
Mathematics, Cetinjski put b.b. 81000 Podgorica, Montenegro}
\email{davidkalaj@gmail.com}



%\email{saksman@mappi.helsinki.fi}
%\date{}
% \subjclass {Primary 30C55, Secondary 31C05}
%\dedicatory{This paper is dedicated to our authors.}
%\keywords{Planar harmonic mappings, Quasiconformal mapping,
%  Convex domains, Rado-Kneser-Choquet theorem}

\maketitle

%\def\thefootnote{}
%\footnotetext{
%\texttt{\tiny File:~\jobname .tex,
%          printed: \number\year-\number\month-\number\day,
%          \thehours.\ifnum\theminutes<10{0}\fi\theminutes}
%}
\makeatletter\def\thefootnote{\@arabic\c@footnote}\makeatother

\begin{abstract}
We prove the following generalization of Schwarz lemma for harmonic mappings. If $u$ is a harmonic mapping of the unit ball $B_n$ onto itself such that $u(0)=0$ and $\|u\|_p:=\left(\int_S|u(\eta)|^pd\sigma(\eta)\right)^{1/p}<\infty$, $p\ge 1$ then  $|u(x)|\le g_p(|x|)\|u\|_p$ for some smooth sharp function $g_p$ vanishing in $0$. Moreover we provide sharp constant $C_p$ in the inequality $\|Du(0)\|\le C_p\|u\|_p$.  Those two results extend some known result from harmonic mapping theory (\cite[Chapter~VI]{ABR}).
\end{abstract}

\maketitle
%\tableofcontents

\section{Introduction}\label{intsec}
On the paper $\mathbf{R}^m$ is the standard Euclidean space with the norm $|x|=\sqrt{\sum x_i^2}$.
Let $p\ge 1$ and assume that $H^p$ is the Hardy space of the  holomorphic mappings on the unit ball $B_n\subset \mathbf{C}^n\cong \mathbf{R}^{2n}$.
In their classical paper \cite{maro}, Macintyre and  Rogosinski proved the following result: Let $f$ be holomorphic on the unit disk such that $f(0)=0$ and such that $\|f\|_{H^p}<\infty$ for $p\ge 1$, then \begin{equation}\label{mecroz}|f(z)|\le \frac{|z|}{(1-|z|^2)^{1/p}}\|f\|_{H^p}\end{equation} with extremal functions $f(w)=\frac{A w}{(1-\bar z w)^{2/p}}$. This is a generalization of Schwarz lemma (for $p=\infty$ it coincides with the classical Schwarz lemma). Then for holomorphic mappings on the unit ball $B_n\subset \mathbf{C}^n$ we have the following result of Zhu \cite[Theorem~4.17]{zhu}: \begin{equation}\label{zhu}|f(z)|\le \frac{1}{(1-|z|^2)^{n/p}}\|f\|_p.\end{equation} Let us sketch the proof of \eqref{zhu}, which imply \eqref{mecroz}, for $n=1$. By definition \begin{equation}\label{zhuk}\|f\|^p_p:=\|f\|_{H^p}^p=\int_S|f(\eta)|^pd\sigma(\eta).\end{equation}
 Here as $d\sigma$ is the normalized rotationally invariant Borel measure on the unit sphere $S=S_{n}=\partial B_n$.
 Choose the holomorphic change $\eta(w)=\varphi_z(w)$ in \eqref{zhuk}, where $\varphi_z$ is an automorphism of the unit ball such that $\varphi(0)=z$. Now by using holomorphic mapping \begin{equation}F_r(w)=f_r(\varphi_z(w))\left(\frac{1-|z|^2}{(1-\left<z,w\right>)^2}\right)^{n/p},\end{equation} and by making use of the mean value inequality we obtain \eqref{zhu}. In order to derive \eqref{mecroz} with $n=1$, from \eqref{zhu}, just alike for the classical proof of Schwarz lemma we make use of holomorphic mapping $g(z)=f(z)/z$, whose $H^p$ norm coincides with the $H^p$ norm of $f$. However an analogous inequality for higher-dimensional case ($n>1$) cannot be proved in the same way. In this paper we will attack  this problem for the class of harmonic mappings, which contain holomorphic mappings.

Assume now that  that $1\le p\le \infty$ and let $1/p+1/q=1$ and consider the Hardy class $\mathcal{H}^p$ of harmonic mappings defined in the unit ball, i.e. of harmonic  mappings $f:B^n\to \mathbf{R}^m$ with  $$\|f\|_p:=\sup_{r}\left(\int_S|f(r\eta)|^pd\sigma(\eta)\right)^{1/p}<\infty.$$ Here as before  $d\sigma$ is the normalized rotationally invariant Borel measure on the unit sphere $S=S^{n-1}$.

It is well known
that a harmonic function (and a mapping)  $u\in \mathcal{H}^p(B)$, $p>1$, where $B=B^n$ is the unit ball with the boundary $S=S^{n-1}$,  has the following integral representation
\begin{equation}\label{poi}u(x)=\mathcal{P}[f](x)=\int_{S^{n-1}}P(x,\zeta)f(\zeta)d\sigma(\zeta),\end{equation}
where
$$P(x,\zeta)=\frac{1-|x|^2}{|x-\zeta|^n}, \zeta\in S^{n-1} $$
is Poisson kernel and $\sigma$ is the unique
normalized rotation invariant Borel measure on $S^{n-1}$ and $|\cdot |$ is the Euclidean norm.

Let us formulate the classical Schwarz lemma for harmonic mappings on the unit ball $B^n\subset \mathbf{R}^n$ and assume its image is $\mathbf{R}^m$. Let $f$ be harmonic on the unit ball, and assume that $\|f\|_\infty<\infty$ and that $f(0)=0$, then we have the following sharp inequality
$$|f(x)|\le U(rN)\|f\|_\infty.$$  Here  $r=|x|$, $N=(0,\dots,0,1)$ and $U$ is a harmonic function of the unit ball into $[-1,1]$ defined by \begin{equation}\label{uext}U(x)= \mathcal{P}[\chi_{S^+}-\chi_{S^-}](x),\end{equation} where $\chi$ is the indicator function and $S^+=\{x\in S: x_n \ge 0\},$ $S^-=\{x\in S: x_n \le 0\}.$

Assume now that $p<\infty$.  We are going to find a sharp function $g(r)$ satisfying the condition $g(0)=0$ in the inequality $$|f(x)|\le g_p(r)\|f\|_p, \ \ \ f\in \mathcal{H}^p,\ \  f(0)=0.$$
\section{The main result}
We prove the following theorem
\begin{theorem}
 Let $p\ge 1$, and let $q$ be its conjugate  and define
 $$g_p(r)=\inf_{a\in [0,\infty)}\left(\int_S |P_r(\eta)-a|^qd\sigma(\eta)\right)^{1/q}.$$  Then for $1\le p<\infty $, $g_p:[0,1)\to [0,\infty)$  is a smooth increasing diffeomorphism  with $g_p(0)=0$,  and for every $f\in \mathcal{H}^p$ with $f(0)=0$, we have
\begin{equation}\label{pri}|f(x)|\le g_p(|x|)\|f\|_p\end{equation} and \begin{equation}\label{sec}\|Df(0)\|\le n  \left(\frac{\Gamma\left[\frac{n}{2}\right] \Gamma\left[\frac{1+q}{2}\right]}{\sqrt \pi\Gamma\left[\frac{n+q}{2}\right]}\right)^{\frac{1}{q}}\|f\|_p.\end{equation} Both inequalities \eqref{pri} and \eqref{sec} are sharp. For $p=\infty$, we have $g_\infty(r)=U(rN)$ which coincides with Schwarz lemma and $g_\infty$  is an increasing diffeomorphism of $[0,1]$ onto itself.
Here $Df(0):\mathbf{R}^n\to\mathbf{R}^m$ is the formal derivative and $\|Df(0)\|=\sup_{|h|=1}|Df(0)h|$.
\end{theorem}
\begin{remark}
It seems unlikely that we can explicitly express the function $g_p(r)$ for general $p$, however we demonstrate some special cases $p=1,2,\infty$ in Section~\ref{spc}, where among the other fact we prove the last part of this theorem. For some sharp pointwice estimates for the first derivative of harmonic mapping we refer to papers \cite{kha,km,km1, kavu}. Some optimal estimates of the harmonic function defined in the unit ball has been obtained in \cite{kam1}, but no normalization  $f(0)=0$ is imposed, so the obtained inequalities in \cite{kam1} are not sharp in the context of this paper.
\end{remark}
\begin{proof}
Let $x\in B_n$ and $\eta\in S^{n-1}$ with $r=|x|$ and let $$P_x(\eta)=\frac{1- r^2}{|x-\eta|^n}$$ and define
$$P_r(\eta) = \frac{1-r^2}{(1+r^2-2r \eta_n)^{n/2}}.$$
Then $$f(x)=\int_S P_x(\eta)f(\eta)d\sigma(\eta),$$ where $S=S^{n-1}$ is the unit sphere. Now since $f(0)=0$, it follows that $$\int_S f(\eta)d\sigma(\eta)=0, $$ and so
$$f(x)=\int_S (P_x(\eta)-a)f(\eta)d\sigma(\eta),$$ for a number $a=a(r)$ not depending on $\eta$.

Hence by using H\"older inequality, and unitary transformations of the unit sphere we have \begin{equation}\label{holder}|f(x)|\le \left(\int_S |P_r(\eta)-a|^qd\sigma(\eta)\right)^{1/q}\|f\|_p. \end{equation} So we are going to consider the minimum of the following function $$\Phi_r(a)=\left(\int_S|P_r(\eta)-a|^q d\sigma(\eta)\right)^{1/q}.$$ Then $\Phi_r(a)$ is convex and it satisfies the conditions $\Phi_r'(0)<0$ and $\Phi_r(\infty)=\infty$. This implies that there is a unique constant $a^*=a(r)\in(0,\infty)$ such that $$\Phi_r(a^*)=\min_{a\in \mathbf{R}}\Phi_r(a).$$ To show that $\Phi_r$ is convex, observe the following simple fact $\Phi_r(a)=\|P_r-a\|_{L^q}$. So $$\Phi_r(\lambda a+(1-\lambda)b)\le \|\lambda(P_r-a)\|_{L^q}+\|(1-\lambda)(P_r-b)\|_{L^q}=\lambda \Phi_r(a)+(1-\lambda) \Phi_r(b).$$ In order to prove that $\Phi_r'(0)<0$, by calculation we find out that $$\Phi_r'(0)=-\int_S |P_r|^{q-1}d\sigma(\eta)\left(\int_S |P_r|^{q}d\sigma(\eta)\right)^{1/q-1}<0.$$
Furthermore
\begin{equation}\label{fip}\Phi_r'(a)=-\left(\int_S |P_r(\eta)-a|^{q}d\sigma(\eta)\right)^{1/q-1}F(r,a),\end{equation} where
$$F(r,a)=\int_S (P_r(\eta)-a)|P_r(\eta)-a|^{q-2} d\sigma(\eta).$$
So
$$F_a(r,a)=(q-1)\int_S |P_r(\eta)-a|^{q-2} d\sigma(\eta)>0,$$ and this implies in particular that $F(r,a)$ as a function of $a$ is strictly increasing, so $\Phi_r$ has only one stationary point which is its minimum which we denote by $a^*=a(r)$.

Since \begin{equation}\label{fstac}F(r,a(r))=0,\end{equation} the implicit function theorem  implies that there is a smooth function $a^*$ depending on $r$ such that $$\frac{\partial a(r)}{\partial r}=-\frac{\frac{\partial F}{\partial r}}{\frac{\partial F}{\partial a}}.$$

Thus the function $g_p(r)=\Phi_r(a(r))$ is smooth function of $r$ as a composition of smooth functions. It satisfies the condition $g_p(0)=0$ and we have $$|f(x)|\le g_p(|x|)\|f\|_{L^p}.$$

Moreover $$g_p'(0)=\left(\int_S|n\eta_n |^qd\sigma(\eta)\right)^{1/q}=n  \left(\frac{\Gamma\left[\frac{n}{2}\right] \Gamma\left[\frac{1+q}{2}\right]}{\sqrt \pi\Gamma\left[\frac{n+q}{2}\right]}\right)^{\frac{1}{q}}$$

So $$\|Df(0)\|\le n  \left(\frac{\Gamma\left[\frac{n}{2}\right] \Gamma\left[\frac{1+q}{2}\right]}{\sqrt \pi\Gamma\left[\frac{n+q}{2}\right]}\right)^{\frac{1}{q}}$$

Since
$g_p(r)=\Phi(r,a(r))$ we have $$\partial_r g_p(r)=\partial_1\Phi(r,a(r))+\partial_2 \Phi(r,a(r))\partial_r a(r)=\partial_1\Phi(r,a(r)).$$
Since
$$\Phi(r,a)=\left(\int_S\left|P_r(\eta)-a\right|^q d\sigma(\eta)\right)^{1/q},$$ we have that

$$\partial_1 \Phi(r,a)= \int_S \partial_r P_r(\eta)(P_r(\eta)-a)|P_r(\eta)-a|^{q-2}d\sigma(\eta)\left(\int_S\left|P_r(\eta)-a\right|^q d\sigma(\eta)\right)^{1/q-1}.$$ Since $$P_r(\eta) = \frac{1-r^2}{(1+r^2-2 r \eta_n)^{n/2}},$$ it follows that $$P_r(\eta)-a=\frac{1-r^2}{(1+r^2-2 r \eta_n)^{n/2}}-a = n\eta_n r + (1-a)+O(r^2),$$ and hence $$\partial_1 \Phi(r,1)=O(r)+\left(\int_S|n\eta_n |^qd\sigma(\eta)\right)^{1/q}.$$
Next we have $\lim_{r\to 0}\partial_1\Phi(r,a(r))= \partial_a\Phi(0,1)$, and so $$\partial_r g_p(0) = \left(\int_S|n\eta_n |^qd\sigma(\eta)\right)^{1/q}.$$

 In order to show that the inequality is sharp, fix $x$ and without loosing the generality assume that $x=RN$, with $R=|x|$ and let $$f_R(\eta)=|P_R(\eta)-a(R)|^{q/p}\mathrm{sign}(P_R(\eta)-a(R))$$ and let $u_R(y)=P[f_R](y)$. From \eqref{fstac} $$F(R,a)=\int_S (P_R(\eta)-a(R))|P_R(\eta)-a(R)|^{q-2} dt=0.$$ Hence we obtain that $u_R(0)=0$. Now, the H\"older inequality \eqref{holder} is an equality for $u_R$ in $x$. This implies the sharpness of inequality. In order to prove that for $p<\infty$, $\lim_{r\to 1}g_p(r)=\infty$, let $f\in \mathcal{H}_p\setminus  \mathcal{H}_\infty$ and assume without loosing of generality that $f(0)=0$. Then $$\sup_{r}g_p(r)\ge \frac{\sup_x|f(x)|}{\|f\|_p}=\infty.$$ Finally prove that $g_p$ is strictly increasing. Let $r<s$ and choose $\|f\|_p=1$ such that $g_p(r)=|f(x_0)|=\max_{|x|\le r}|f(x)|$. Clearly $f$ is not a constant function. Then by maximum principle $|f(x_0)|<\max_{|x|=s}|f(x)|\le g_p(s)$. So $g_p$ is a strictly increasing function.

  The last part of the proof, i.e. the case $p=\infty$, follows from the previous proof and the next section.
 \end{proof}
 As a corollary of our main result we obtain
 \begin{corollary}
 If $f\in \mathcal{H}^2$, then $\|Df(0)\|\le \sqrt{n}\sqrt{\|f\|^2_2-|f(0)|^2|}$.
 \end{corollary}
 \begin{proof}
 Let $g(x)=f(x)-f(0)$, then $Df(0)=Dg(0)$, on the other hand $$\|g\|_2^2=\left<f-f(0),f-f(0)\right>=\|f\|^2+|f(0)|^2-2\left<f,f(0)\right>=\|f\|_2^2-|f(0)|^2.$$  On the other hand
 $$|Dg(0)|\le \lim_{r\to 0} g_2'(r)\|g\|_2= \sqrt{n}\|g\|_2.$$ The result follows.
 \end{proof}

 \section{Special cases}\label{spc}

\subsection{The case $p=\infty$}
In this case we deal with the extremal problem $$\inf_a \int_S |P_r(\eta)-a| d\sigma(\eta).$$ Let $a_0=\frac{1-r^2}{(1+r^2)^{n/2}}$. Then
\[\begin{split}\int_S |P_r(\eta)-a_0| d\sigma(\eta)&=\int_{S^+} P_r(\eta)d\sigma(\eta)-\int_{S^-} P_r(\eta)d\sigma(\eta)
\\&=\int_{S^+} (P_r(\eta)-a)d\sigma(\eta)-\int_{S^-} (P_r(\eta)-a)d\sigma(\eta)\\&\le \inf_a \int_S |P_r(\eta)-a_0| d\sigma(\eta).\end{split}\] So $a^*=\frac{1-r^2}{(1+r^2)^{n/2}}.$

This implies that $$\|f(z)\|\le U(rN)\|f\|_\infty,$$ which is known as the classical Schwarz lemma for harmonic mappings.
\subsection{The case $p=2$}
In this case we deal with the extremal problem $$g_p(r)=\left(\inf_a \int_S |P_r(\eta)-a|^2 d\sigma(\eta)\right)^{1/2}.$$ We have (see \cite[p.~140]{ABR})
$$\int_S |P_r(\eta)-a|^2 d\sigma(\eta)=\int_S P^2_r(\eta) d\sigma(\eta)+a^2\int_S  d\sigma(\eta)-2 a \int_S P_r(\eta) d\sigma(\eta)$$

$$=\frac{1-|x|^4}{(1-2|x|^2+|x|^4)^{n/2}}+a^2-2a.$$ So $a^*=1$ and $$\left(\inf_a \int_S |P_r(\eta)-a|^2 d\sigma(\eta)\right)^{1/2} =\sqrt{\frac{1+|x|^2}{(1-|x|^2)^{n-1}}-1}.$$
This implies that $$|f(x)|\le\left( \sqrt{\frac{1+r^2}{(1-r^2)^{n-1}}-1}\right)\|f\|_2$$ which coincides with analogous statement in \cite[p.~140]{ABR}.
\subsection{The case $p=1$}
In this case we deal with the extremal problem $$g_p(r)=\inf_a \sup_\eta |P_r(\eta)-a|.$$ Since $$\max_\eta P_r(\eta)=\frac{1-r^2}{(1-r)^n}$$ and
 $$\min_\eta P_r(\eta)=\frac{1-r^2}{(1+r)^n}, $$ we easily conclude that $$g_p(r)=\frac{1}{2}\left(\frac{1-r^2}{(1-r)^n}-\frac{1-r^2}{(1+r)^n}\right)$$
(In this case $a^* =\frac{1}{2}\left(\frac{1-r^2}{(1-r)^n}+\frac{1-r^2}{(1+r)^n}\right).$)
So $$|f(x)|\le \frac{1}{2}\left(\frac{1-r^2}{(1-r)^n}-\frac{1-r^2}{(1+r)^n}\right)\|f\|_1.$$

\end{document}